  \documentclass[11pt]{article}
\usepackage{latexsym,amsfonts,amssymb,amsmath,amsthm}

\usepackage{graphicx}
\usepackage[utf8]{inputenc}
\usepackage{hyperref}
\begin{document}

\newtheorem{tm}{Theorem}[section]
\newtheorem{prop}[tm]{Proposition}
\newtheorem{defin}[tm]{Definition}
\newtheorem{coro}[tm]{Corollary}
\newtheorem{lem}[tm]{Lemma}
\newtheorem{assumption}[tm]{Assumption}
\newtheorem{rk}[tm]{Remark}
\newtheorem{nota}[tm]{Notation}
\numberwithin{equation}{section}
\numberwithin{tm}{section}


\newcommand{\eqd}{\sim}
\def\p{\partial}
\def\R{{\mathbb R}}
\def\N{{\mathbb N}}
\def\Q{{\mathbb Q}}
\def\C{{\mathbb C}}
\def\l{{\langle}}
\def\r{\rangle}
\def\t{\tau}
\def\k{\kappa}
\def\a{\alpha}
\def\la{\lambda}
\def\De{\Delta}
\def\de{\delta}
\def\ga{\gamma}
\def\Ga{\Gamma}
\def\ep{\varepsilon}
\def\si{\sigma}
\def\Re {{\rm Re}\,}
\def\Im {{\rm Im}\,}
\def\E{{\mathbb E}}
\def\P{{\mathbb P}}
\def\Z{{\mathbb Z}}
\def\D{{\mathbb D}}
\newcommand{\ceil}[1]{\lceil{#1}\rceil}

\title{Dynamics of a parabolic-ODE competition system in heterogeneous environments}


\author{ Yuan Lou\footnote{lou@math.ohio-state.edu; Research  partially supported by NSF grant DMS-1853561}
\qquad Rachidi B. Salako\footnote{salako.7@osu.edu}
\\
\\
{\small Department of  Mathematics, Ohio State University,}\\
{\small Columbus, OH 43210, USA}
%
}

\date{} 
\maketitle

\begin{abstract} This work is concerned with the large time 
behavior of the solutions of a parabolic-ODE hybrid system,
modeling the competition of two populations which
are identical except their movement behaviors:
one species moves by random dispersal
while the other does not diffuse. We show
that the non-diffusing population will always
drive the diffusing one to extinction in 
environments with sinks.
In contract, the non-diffusing
and diffusing populations can coexist in
environments without sinks.

\end{abstract}


\smallskip


\section{Introduction} 

In this work we investigate the asymptotic 
behaviors of classical solutions of the following 
parabolic-ODE competition system:
\begin{equation}\label{Main-eq1}
\begin{cases}
u_t=d\Delta u+ u(a(x)-u- v) & x\in\Omega, \ t>0,\cr
v_t =v(a(x)-u-v) & x\in\Omega, \ t>0,\cr
\frac{\partial u}{\partial n}=0 & x\in\partial\Omega, \ t>0,\cr
u(x,0)=u_0(x) \quad \textrm{and} \quad v(x,0)=v_0(x) & x\in\Omega,
\end{cases}
\end{equation}
where $d>0$ is a constant, $a(\cdot)\in C^{\alpha}(\overline{\Omega})$
for some $\alpha\in (0,1)$, and $\Omega$ is a bounded domain in $\mathbb{R}^N$ with smooth boundary $\partial\Omega$, $N\geq 1$.  
The functions $u(x,t)$ and $v(x,t)$ denote the density functions of two 
populations residing in the same habitat $\Omega$ and competing for 
a common limited resource. Hence, the initial conditions $u_0(x)$ and $v_0(x)$ 
are assumed to be non-negative, not identically zero and
continuous functions 
on $\overline{\Omega}$.  The function $a(x)$ represents their common intrinsic growth rate, and throughout this paper it is assumed to be 
non-constant and positive somewhere in $\Omega$, reflecting 
that the environment is heterogeneous in space. 
The region $\{x\in \bar\Omega: a(x)<0\}$ is referred to as 
the sink (low quality habitat),  where the growth rate 
of the population is negative. $n$ denotes the 
outward unit normal vector on $\partial\Omega$, and
the boundary condition for $u$ means that no individuals
cross the boundary.

We note that system \eqref{Main-eq1} is a  
Lotka-Volterra competition model, 
in which the species $v(x,t)$ has zero diffusion rate, while  $u(x,t)$ 
has a 
positive diffusion rate. In the recent years there has been  
increasing interest in the dynamics of
two-species Lotka–Volterra competition models in heterogeneous
environments; see \cite{CC2003, CC2018, Cosner2014,
HeNi3, HeNi4, Hutson2001, LamNi2012, lou2006, Lou2008, Ni2011}
 and the references therein. 
To motivate our work, 
consider the following fully-parabolic Lotka-Volterra competition system:
\begin{equation}\label{Main-eq1-epsilon-positive}
\begin{cases}
u_t=d\Delta u+ u(a(x)-u- v) & x\in\Omega, \ t>0,\cr
v_t =\varepsilon\Delta v +  v(a(x)-u-v) & x\in\Omega, \ t>0,\cr
\frac{\partial u}{\partial n}=0 & x\in\partial\Omega, \ t>0,\cr
u(x,0)=u_0(x) \quad \textrm{and} \quad v(x,0)=v_0(x) & x\in\Omega,
\end{cases}
\end{equation}
where $\varepsilon, d>0$ are constant positive numbers. 
When $\varepsilon=0$, system \eqref{Main-eq1-epsilon-positive} reduces to system \eqref{Main-eq1}. 
Hence the system studied in this work can be regarded as the limiting case of \eqref{Main-eq1-epsilon-positive} by formally letting $\varepsilon\to0$. 
Concerning the large time behavior of
the solutions of \eqref{Main-eq1-epsilon-positive}, it is well known that the relationship between $d$ and $\varepsilon$ plays an important role. For our interest, we suppose that $0<\varepsilon<d$. In \cite{DHMP1998}, Dockery et. al. showed that the species with the smaller diffusion rate is always favored by the competition, provided that $a(x)$ is non-constant. More precisely, if we assume that $a(x)$ is non-constant and that in the absence of competition, there exist two semi-trivial steady states of \eqref{Main-eq1-epsilon-positive}, denoted as $(u^*_{d}(x),0)$ and $(0,v_{\varepsilon}^*(x))$,  with $0<\min\{u^*_{d}(x),v_{\varepsilon}^*(x)\}<\max\{u^*_{d}(x),v_{\varepsilon}^*(x)\}<\infty$, then 
$(0,v^*_{\varepsilon}(x))$ is globally asymptotically stable. 
This conclusion is often referred to as the evolution of slow dispersal
\cite{Hastings1983}.
It is thus natural to inquire  about the dynamics of 
\eqref{Main-eq1-epsilon-positive} when $\varepsilon=0$,
in particular, whether the non-diffusing population 
is still able to drive the diffusing one to extinction,
as in the case of $\epsilon\in (0, d)$. 

\noindent To state our main result on system \eqref{Main-eq1}, we first introduce
a few notations. Let $C(\overline{\Omega})$ denote the Banach space of uniformly continuous functions on $\Omega$ endowed with the usual sup-norm, 
and $[C(\overline{\Omega})]^+$ denotes the closed subspace of $C(\overline{\Omega})$ consisting of non-negative functions. 
\begin{defin}\label{def1}
For given $u_0,v_0\in C(\overline{\Omega})$ with $u_0(x)\ge 0$ and $v_0(x)\geq 0$ and $T\in(0,\infty]$, we say that $ (u(x,t;u_0,v_0),v(x,t;u_0,v_0))$ is a classical solution of \eqref{Main-eq1} on $[0,T)\times\Omega$ with $ (u(x,0;u_0,v_0),v(x,0;u_0,v_0))=(u_0(x),v_0(x)) $
if the followings hold:
\begin{itemize}
\item[\rm 1)]for every $p>N$, \begin{equation}\label{eq:Regulaity-1-intro}
    u\in  C([0,T):[C(\overline{\Omega})]^+)\cap C^1((0,T):C(\overline{\Omega}))\cap C((0,T):W^{2,p}(\Omega)),
\end{equation}
\item[\rm 2)] \begin{equation}\label{eq:Regulaity-2-intro}
    v\in  C([0,\infty):[C(\overline{\Omega})]^+)\cap C^1((0,\infty):C(\overline{\Omega})), 
    \end{equation}
\item[\rm 3)] \begin{equation} \lim_{t\to0}\left[\|u(\cdot,t;u_0,v_0)-u_0\|_{\infty}+\|v(\cdot,t;u_0,v_0)-v_0\|_{\infty}\right]=0,
\end{equation}
\item[\rm 4)] $(v(x,t;u_0,v_0))$ satisfies the second equation of \eqref{Main-eq1} in classical sense,
\item[\rm 5)]  $u(x,t;u_,v_0)$ is a solution of the first equation of \eqref{Main-eq1} as an abstract evolution equation in $C(\overline{\Omega})$.
\end{itemize}
By a global classical solution of \eqref{Main-eq1}, we mean a classical solution  on $[0,\infty)\times\Omega$.
\end{defin}
\noindent We first note from the continuous embedding results of $W^{2,p}(\Omega)$ in $C^{1+\beta}(\Omega)$ for $0<\beta\ll 1$ and $p>N$, that for any solution $(u(x,t;u_0,v_0),v(x,t;u_0,v_0))$ of \eqref{Main-eq1} in the sense of Definition \ref{def1} that both $\partial_tu(x,t;u_0,v_0)$ and $\partial_xu(x,t;u_0,v_0)$ exist in classical sense and are continuous.  When $\|u_0\|_{\infty}>0$, it follows from the comparison principle for parabolic equations that $ u(x,t;u_0,v_0)>0 $ for all $t>0$ and $x\in\Omega$ and that $ v(x,t;u_0,v_0)>0 $ for every $t>0$ and $x\not\in\{y\in\Omega: \ v_0(y)=0\}$. 
The following result addresses the existence of global classical solution of \eqref{Main-eq1}:

\begin{prop}\label{Tm-01} Given any $u_0,v_0\in[C(\overline{\Omega})]^+$, system \eqref{Main-eq1} has a unique global classical solution on $[0,\infty)\times\Omega$. Moreover, it holds that 
\begin{equation}\label{Tm-01-eq1}
    \limsup_{t\to\infty}\max\{\|u(\cdot,t;u_0,v_0)\|_{\infty},\|v(\cdot,t;u_0,v_0)\|_{\infty}\}\leq \|a\|_{\infty}.
\end{equation}
\end{prop}

\noindent 
When $a(x)$ changes sign in $\Omega$, the following result provides
a rather complete feature of the behavior of solutions for large time:

\begin{tm}\label{Main-tm-1} Suppose that 
\begin{equation}\label{H2}
\{x\in\overline{\Omega} \ : \ a(x)\le 0\}\neq \emptyset.
\end{equation} 
For every  non-negative initial $u_0,v_0\in C(\overline{\Omega})$ satisfying
\begin{equation}\label{eq:initial-cond1}
   \{x\in\overline{\Omega}\ :\ v_0(x)=0\}\subset \{x\in\overline{\Omega} \ : \ a(x)\le 0\},
\end{equation}
we have that 
\begin{equation}\label{asymptotic-eq}
\lim_{t\to\infty}(u(x,t;u_0,v_0),v(x,t;u_0,v_0))=(0,a_{+}(x)),\quad \forall x\in\Omega.
\end{equation}
Moreover the convergence $u(x,t;u_0,v_0)\to 0 $ as $t\to\infty$ is uniform in $x\in\Omega$.
\end{tm}

Biologically, Theorem \ref{Main-tm-1}
implies that the non-diffusing population will always
drive the diffusing one to extinction in 
environments with sink.
When (\ref{H2}) fails to hold, i.e. if 
$a(x)$ is strictly positive, Theorems \ref{Tm-02-0},
\ref{tm-2}, and \ref{tm-positive} provide
some partial answers on
the large time behaviors of solutions, 
which illustrate that  the non-diffusing
and diffusing populations can coexist in
environments without sinks.
We also note that Theorem \ref{Main-tm-1}
can not hold without the condition  (\ref{eq:initial-cond1}).

\section{Preliminaries}

In this section we recall a few results from the literature on the single species model that will be needed for our discussion and prove Theorem \ref{Tm-appriori-result} (see below). It turns out that Theorem \ref{Tm-appriori-result} will be 
essential for our proof of Theorem \ref{Main-tm-1}.

\noindent We first consider the single species  equation 
\begin{equation}\label{Main-eq2}
\begin{cases}
u_t=d \Delta u+ u(a(x)-u)\quad  & x\in\Omega, \ t>0,\cr
\frac{\partial u}{\partial n}=0 \quad & x\in\partial\Omega, \ t>0,\cr
u(x,0)=u_0(x)\quad  & x\in\Omega.
\end{cases}
\end{equation}
We denote by $u(x,t;u_0)$ the classical solution of \eqref{Main-eq2}. Next, we let $u^*(x)$ denotes the unique non-negative steady solution 
of \eqref{Main-eq2} attracting all positive solutions of \eqref{Main-eq2}. 
The existence of $u^*(x)$ is well known; see \cite{CC2003}. 
Throughout this section we shall suppose that $u^*(x)>0$. 
Hence, $u^*(x)$ is the only solution of the elliptic equation
\begin{equation}\label{eq:steady-state}
\begin{cases}
0=d \Delta u^* +u^*(a(x)-u^*)\quad & x\in\Omega,\cr
u^*(x)>0 \quad & x\in\Omega,\cr
\frac{\partial u^*}{\partial n}=0 \quad & x\in\partial\Omega.
\end{cases}
\end{equation}
Hence, for every non-negative and not identically zero initial function 
$u_0\in C(\overline{\Omega})$, it holds that
\begin{equation}\label{eq:single-species-stability}
\lim_{t\to\infty}\|u(t,\cdot;u_0)-u^*(\cdot)\|_{\infty}=0.
\end{equation} 
Note that a sufficient condition to ensure that $u^*(x)>0$ is to require that $\int_{\Omega}a(x)dx>0$. Note also that $u(x,t;u^*)=u^*(x)$ for every $x\in\Omega$, $t\geq 0$. It is important to note that $u^*(x)$ is not a constant function. Indeed, otherwise, by \eqref{eq:steady-state}, we would have that $a(x)=u^*(x)$ 
for all $x\in\Omega.$
 So, $a(x)$ must also be a constant function, which contradicts 
 our standing assumption on $a(x)$.  We note that the constant solution 
$u(x,t):=\|a\|_{\infty},$ is a super-solution of \eqref{Main-eq2}, hence by \eqref{eq:single-species-stability}, the fact that $u^*(x)$ is not a constant function, and the comparison principle for scalar parabolic equations, we must have that 
 \begin{equation*}
 u^*(x)=\inf_{t>0}u(t,x;\|a\|_{\infty})<\|a\|_{\infty},\quad x\in\Omega.
 \end{equation*}
 Thus, since $u(x):=\|a\|_{\infty}$ is a super-solution of \eqref{eq:steady-state}, it follows by Hopf  boundary lemma that 
 \begin{equation}\label{eq:00}
 \max_{x\in\Omega}u^*(x)<\|a\|_{\infty}.
 \end{equation}
 
  The following result holds.
\begin{lem}\label{lemma1}
There holds that 
\begin{equation}\label{eq:01}
\int_{\Omega}u^*(x)(a(x)-u^*(x))dx=0,
\end{equation}
and 
\begin{equation}\label{eq:02}
\Omega^*:=\{x\in\Omega\ :\ a(x)-u^*(x)>0  \}\neq \emptyset.
\end{equation}
\end{lem}
\begin{proof}
Integrating the first equation in \eqref{eq:steady-state}, then use integration by part formula yields \eqref{eq:01}. Observe that \eqref{eq:02} easily follows from \eqref{eq:00}.
\end{proof}

Next, we consider the sequence $\{u^*_{k}\}_{k\geq 0}$ with $u^*_0(x)=u^*(x)$ defined as follows. Suppose that $u^*_{k}(x)$ is being defined, we let $u^*_{k+1}(x)$ denotes the unique non-negative attracting set of solutions of the PDE
\begin{equation}\label{d:3}
\begin{cases}
u_t=d \Delta u +u(a(x)-\left[a(x)-u^*_k(x)\right]_+-u),\quad & x\in\Omega,\ t>0,\cr
\frac{\partial u}{\partial n}=0\quad & x\in\partial\Omega,\ t>0,
\end{cases}
\end{equation}
where we adopt the conventional notation $[a(x)-u^*_k(x)]_{+}=\max\{0, a(x)-u^*_k(x)\}$.
Note that $u^*_{k+1}(x)$ is a non-negative steady state solution of \eqref{d:3} for every $k\geq0$. Let $u_{k+1}(t,x;w)$ denotes the solution of \eqref{d:3} with $u_{k+1}(0,x)=w(x)$. We prove the following result. 

\begin{tm}\label{Tm-appriori-result} Let $\{u^*_{k}\}_{k\geq 0}$ be defined as above.
\begin{itemize}
    \item[\rm(i)] $u^*_{k+1}(x)\leq u^*_{k}(x)$ for every $x\in\Omega$ and $k\geq 0$.
    \item[\rm(ii)] It holds that 
    \begin{equation}
        \lim_{k\to\infty}\|u^*_k-[a_{\min}]_{+}\|_{\infty}=0,
    \end{equation}
    where 
    $a_{min}:=\min_{x\in\overline{\Omega}}a(x)$
    and $[a_{\min}]_{+}=\max\{0,a_{min}\}$. 
    
\end{itemize}

\end{tm}

\begin{proof}
$(i)$ We prove this by induction. We first note $u_0^*(x)=u^*(x)$ is super-solution of \eqref{d:3} with $k=0$ and initial 
condition $u^*(x)$. Hence, 
by the comparison principle for parabolic equations, we obtain that  
$$ 
u^*_1(x)=\lim_{t\to\infty}u_{1}(t,x;u^*_0)\leq u_1(1,x;u^*_0) \le u^*_{0}(x), \quad \forall\ x\in\Omega.
$$
Suppose by induction hypothesis that $(i)$ holds for $k=1,\cdots,m$, with $m\geq 1$. 
As in the previous case we note that $u^*_m(x)$ is also 
a super-solution of \eqref{d:3} with $k=m-1$ and initial
condition $u_m^*(x)$ 
because $a(x)-(a(x)-u^*_m(x))_+\leq a(x)-(a(x)-u^*_{m-1}(x))_+$ by induction hypothesis. Therefore similar arguments 
yields that the result also holds for $k=m+1$.

\medskip

$(ii)$ By $(i)$, we have that $u^*_k(x)\to U^*(x)$ as $k\to\infty$ for some $U^*\in L^{\infty}(\Omega)$. Moreover, since $\sup_{k\geq 0}\|u^*_k\|\leq \|a\|_{\infty}$, by estimates for elliptic equations, 
we have that $u_{k}^*(x)\to U^*(x)\in W^{2,p}(\Omega)$ for every $p>N$,
and $U^*(x)$ satisfies
\begin{equation}\label{d:6}
\begin{cases}
0=d\Delta U^*+U^*(a(x)-\left[ a(x)-U^*(x)\right]_{+}-U^*)\quad & x\in\Omega,\cr
\frac{\partial U^*}{\partial n}=0 & x\in\partial\Omega.
\end{cases}
\end{equation}
Observe that $ a(x)-\left[ a(x)-U^*(x)\right]_{+}-U^*=-\left[ a(x)-U^*(x)\right]_{-}$  and integrating the first equation of \eqref{d:6} yields that 
$$ 
\int_{\Omega}U^*\left[ a(x)-U^*(x)\right]_{-}dx=0,
$$  which implies that 
\begin{equation}\label{d:10}
U^*(x)\left[ a(x)-U^*(x)\right]_{-}=0,
\end{equation} 
since $x\mapsto U^*(x)\left[ a(x)-U^*(x)\right]_{-}$ is continuous. Thus, since $U^*(x)\ge 0$, we conclude that 
$$ 
0\leq U^*(x)\leq \max\{0,a(x)\}\quad \forall\ x\in\Omega.
$$ Hence, by \eqref{d:6}, we conclude that $U^*(x)=c^*$ for some non-negative constant $c^*$. So, by \eqref{d:10}, we obtain that $ c^*\leq [a_{\min}]_{+}$. Hence if $[a_{\min}]_{+}=0$, the result follows. So, it remains 
to consider the case $[a_{\min}]_{+}>0$. That is, $[a_{\min}]_{+}=a_{\min}>0$. In this case, it is clear that $u(t,x)=a_{\min}$ is a sub-solution of \eqref{d:3}
with initial condition $a_{min}$. Hence, we have 
$$a_{\min}\leq \lim_{t\to\infty}u_k(x,t;a_{\min})= u^+_{k}(x)\quad 
\forall x\in\Omega, \ k\geq 0. $$
As a result, we have that $a_{\min}\leq U^*(x)=c^*\leq a_{\min}$. This completes the proof.
\end{proof}

\section{Proof of Proposition \ref{Tm-01}}

In this section we present the proof of Proposition \ref{Tm-01}. We start with the local existence of classical solutions.

\begin{lem}\label{Local-existence-tm}
For every  $u_0,v_0\in [C(\overline{\Omega})]^+$, there is a unique $T_{max}>0$ such that \eqref{Main-eq1} has a unique classical solution $(u(x, t;u_0,v_0),v(x,t;u_0,v_0))$ satisfying $(u(x,0;u_0,v_0),v(x,0;u_0,v_0))=(u_0(x),v_0(x))$ on $[0,T_{\max})\times\Omega$. Moreover, if $T_{\max}<\infty$, then 
\begin{equation}\label{eq:extension-crit} 
\lim_{t\to T_{\max}^-}\|u(\cdot,t;u_0,v_0)\|_{\infty}=+\infty.
\end{equation}
\end{lem}
\begin{proof} We first note that if \eqref{Main-eq1} has a local classical solution on some $[0,T)\times\Omega$, the uniqueness and  extension criterion \eqref{eq:extension-crit} follows 
from classical extension argument in the literature. Therefore, we will only show that \eqref{Main-eq1} has a local classical solution.   We used fixed point arguments to prove the result.
Let $R>\|u_0\|_{\infty}$ and $T>0$ be given, and define
 $$ 
 \mathcal{S}_{R,T}:=\{u\in C([0,T]:C(\overline{\Omega})) \, :\, u(x,0)=u_0(x) \, \text{and}\, \|u\|_{\infty}\leq R\}
 $$
endowed with sup-norm.  Next, for every $\lambda>\|a\|_{\infty}+R$ and  $u\in  \mathcal{S}_{R,T}$, define the integral operator 
 \begin{align}\label{u-eq1}
     \mathcal{T}(u)(t)=T_{\lambda,d}(t)[u_0] + \int_0^t T_{\lambda,d}(t-s)[(\lambda+a-u(s)-\mathcal{V}(u)(s))u(s) ]ds
 \end{align}
 where 
 \begin{equation}\label{v-eq1}
    \mathcal{V}(u)(s)=\frac{v_0(x)e^{\int_0^t(a(x)-u(x,s))ds}}{1+v_0(x)\int_0^t\left[e^{\int_0^s(a(x)-u(x,\tau))d\tau}\right]ds} 
 \end{equation}
 and $T_{\lambda,d}(t)$ denotes the analytic $c_0-$semigroup generated by $A_{p}u=d\Delta u-\lambda u$ on $L^p(\Omega)$ with $\text{Dom}(A_p)=\{w\in W^{2,p}(\Omega)\,: \, \frac{\partial w}{\partial n}=0 \,\,\text{on}\,\,\partial\Omega \}$.  It is clear from \eqref{v-eq1} that 
 \begin{equation}
 \mathcal{V}(u)\in C([0,\infty):[C(\overline{\Omega})]^+)\cap C^1((0,\infty):C(\overline{\Omega})) \quad \forall\, u\in \mathcal{S}_{R,T},
 \end{equation}
 with
 \begin{equation}\label{v-eq2}
     \partial_t\mathcal{V}(u)=(a-u-\mathcal{V}(u))\mathcal{V}(u),
 \end{equation}
 and \begin{equation}\label{v-eq3}
     \lim_{t\to\infty}\|\mathcal{V}(u)(t)-v_0\|_{C(\overline{\Omega})}=0.
 \end{equation}
  We note that by considering 
$$
\textrm{Dom}(A):=\{ u \in W^{2,p}(\Omega),\ \, p>N,\ \tfrac{\partial u}{\partial n}=0\, \,\textrm{on}\, \,\partial\Omega,\ Au\in C(\overline{\Omega}) \}
$$
where $Au=d\Delta u-\lambda u$, it is well known
(see \cite[Theorem 2] {Stewart}) that $A$ generates an analytic semigroup on $C(\overline{\Omega})$,  given by $T_{\lambda,p}(t)$. Thus it follows from \eqref{u-eq1} that $\mathcal{T}(u)\in C([0,T]:C(\overline{\Omega}))$ for every $u\in\mathcal{S}_{R,T}$. Note from the choice of $\lambda$ that $u(x,t)\ge 0$. Hence, the maximum principle implies that  $\mathcal{T}(u)\in C([0,T]:[C(\overline{\Omega})]^+)$ for every $u\in\mathcal{S}_{R,T}$.
 
 Next, we show that for $0<T\ll 1$, the map $\mathcal{T}$ maps $\mathcal{S}_{R,T}$ into itself. Indeed, this easily follows from the continuity of the maps $\mathcal{T}$ at $t=0$, since $\lim_{t\to0}\mathcal{T}(u)(t)=u_0$ uniformly in $u\in\mathcal{S}_{R,T}$ and $\|u_0\|_{C(\overline{\Omega})}<R$.
 
 Finally, we claim that the map $\mathcal{T}$ is a contraction for $0<T\ll 1$. To this end, observe that it is enough to show that the map $t\mapsto \mathcal{V}(u)(t)$ is Lipschitz continuous. This in turn follows from the fact that 
 \begin{align}
    & \left|e^{\int_0^t(a(x)-u_1(x,s))ds}-e^{\int_0^t(a(x)-u_2(x,s))ds}\right|\cr
    \leq & \left|\int_0^t(u_1(x,s)-u_2(x,s))ds\right|\sup_{\theta\in[0,1]}
    e^{\int_{0}^t\left[\theta u_1(x,s)+(1-\theta)u_2(x,s)\right]ds}\cr
    \leq & Te^{RT}\|u_1-u_2\|_{\mathcal{S}_{R,T}},
 \end{align}
for every $u_1,u_2\in\mathcal{S}_{R,T}$, where we have used the mean-value theorem. 
Therefore, for $0<T\ll 1$, it follows from the contraction mapping theorem that the map $\mathcal{S}_{R,T}\ni u\mapsto\mathcal{T}(u) \in \mathcal{S}_{R,T}$ has a unique fixed point.

To complete the proof of the regularity of the function $[0,T]\ni t\mapsto  u(\cdot,t)\in C([0,T]:C(\overline{\Omega}))$ we set 
$$
f(t)=(\lambda+a-u(t)-v(t))u(t) 
$$
where $v(t)=\mathcal{V}(u)(t)$ and show that the function $(0,T)\ni t\mapsto f(t)\in C(\overline{\Omega})$ is locally H\"older continuous. Hence, the regularity follows  by \cite[Theorem 1.2.1, Page 43]{Amann}. Again by the regularity of $v(t)$, to show that $f(t)$ is locally H\"older  continuous, it is enough to show that the function \begin{equation}
t\mapsto u(t)\in  C(\overline{\Omega})
\end{equation}
is locally H\"older continuous. Let $0<\alpha<1$ and let $X^{\alpha}$ denotes the fractional power space of $-A_p$. By $L_p-L_q$ estimates,  there exist constants $c_\alpha>0$  and $\omega>0$ such that 
$$ 
\|(T_{\lambda,p}(h)-I)T_{\lambda,p}(t)w\|_{C(\overline{\Omega})}\leq c_{\alpha} h^{\alpha}t^{-\alpha}e^{-\omega t}\|w\|_{C(\overline{\Omega})}\quad \forall w\in C(\overline{\Omega}).
$$
Therefore, for every $0<t<t+h<T$, we have 
\begin{align*}
&\,\|u(t+h)-u(t)\|_{C(\overline{\Omega})}\cr
\leq & \|(T_{\lambda,d}(h)-I)T_{\lambda,d}(t)u_0\|_{C(\overline{\Omega})} +\int_0^t\|(T(h)-I)T(t-s)f(s)\|_{C(\overline{\Omega})}ds\cr
&+\int_t^{t+h}\|T(t+h-s)f(s)\|_{C(\overline{\Omega})}ds\cr
\leq &\, M\left( h^{\alpha}+ h^{\alpha}+h\right),
\end{align*}
where $M$ is a constant depending on $R,\alpha,p$ and $T$. As a result, we have the desired result.
\end{proof}

Next, we complete the proof of Proposition \ref{Tm-01}.

\begin{proof}[Proof of Proposition {\rm\ref{Tm-01}}]
Let $u_0,v_0\in [C(\overline{\Omega})]^+$ be given  and let  $(u(x,t;u_0,v_0),v(x,t;u_0,v_0))$ denotes the classical solution given by Lemma \ref{Local-existence-tm}. Since $v(x,t;u_0,v_0)\geq 0$, we have that 
\begin{equation}\label{p:eq:1} 
u_t\leq d\Delta u+u(a(x)-u),\quad \forall\, x\in\Omega, t\in(0,T_{\max}).
\end{equation}
Hence, the comparison principle implies that 
$u(x,t;u_0,v_0)\leq \max\{\|u_0\|_{\infty},\|a\|_{\infty}\}$
for all
$x\in\Omega$ and  $t\in(0,T_{\max})$.
So, by \eqref{eq:extension-crit}, we conclude that $T_{\max}=+\infty$. Again by \eqref{p:eq:1} and the comparison principle for parabolic equations, we have that 
$\limsup_{t\to\infty}\|u(\cdot,t;u_0,v_0)\|_{\infty}\leq \|a\|_{\infty}.
$
Similar arguments yield that $ 
\limsup_{t\to\infty}\|v(\cdot,t;u_0,v_0)\|_{\infty}\leq \|a\|_{\infty}.
$  
\end{proof}

\section{Proof of Theorem \ref{Main-tm-1}}

In this section, we present the proof of Theorem \ref{Main-tm-1}. Throughout this section we suppose $u_0(x)$ and $v_0(x)$ are chosen fixed and satisfy the assumption of Theorem \ref{Main-tm-1}. We first prove a few preliminary results. 

\begin{lem}\label{lemma3}
It holds that 
\begin{equation}\label{eq:lower-bound-for-v}
\liminf_{t\to\infty}v(x,t;u_0,v_0)\geq [a(x)-u^*(x)]_+,\quad \forall\ x\in\Omega,
\end{equation}
where $u^*(x)$ is given by \eqref{eq:steady-state}. 
In particular, the first species can not drive the second species to extinction. 
\end{lem}
\begin{proof}
Since $v(x,t;u_0,v_0)\ge 0$ for every $x\in\Omega$ and $t\geq 0$, it follows 
that
$
u(x,t;u_0,v_0)\leq u(x,t;u_0)
$
for all $x\in\Omega$ and $t\ge 0$,
where $u(x,t;u_0)$ denotes the unique classical solution of \eqref{Main-eq2}. Thus, by \eqref{eq:single-species-stability}, it holds that
$$ 
\limsup_{t\to\infty}\sup_{x\in\Omega}\left[u(x,t;u_0,v_0)-u^*(x)\right]\leq \limsup_{t\to\infty}\sup_{x\in\Omega}\left[u(x,t;u_0)-u^*(x)\right]=0.
$$
So, for every $\varepsilon>0$ there is $t_\varepsilon\gg 1$ such that 
$u(x,t;u_0,v_0)<u^*(x)+\varepsilon$ for all $x\in\Omega$ and $t\geq t_\varepsilon$.
Whence, from the second equation in \eqref{Main-eq1}, we deduce that 
\begin{equation}\label{eq:03}
v_t\geq v(a(x)-u^*(x)-\varepsilon-v),\quad \forall t\geq t_\varepsilon, x\in\Omega.
\end{equation}
The comparison principle for ODEs  thus implies that 
$$ 
\liminf_{t\to\infty}v(x,t;u_0,v_0)\geq \left[a(x)-u^*(x)-\varepsilon\right]_+,\quad x\in\Omega.
$$
Letting $\varepsilon\to0^+$ in the last inequality leads to \eqref{eq:lower-bound-for-v}. The last statement of Lemma \ref{lemma3} follows from Lemma \ref{lemma1} and inequality \eqref{eq:lower-bound-for-v}.
\end{proof}

Next, we improve the previous Lemma to

\begin{lem}\label{Lem1} Suppose that $u^*(x)>0$.
For every $x_0\in\Omega^*$, it holds that 
\begin{equation*}
\liminf_{t\to\infty}v(x,t;u_0,v_0)>a(x_0)-u^*(x_0)>0.
\end{equation*}
Furthermore,  there exist $0<\varepsilon_0\ll 1$ and $T_0\gg 1$ such that 
\begin{equation}\label{d:9}
u(x,t;u_0,v_0)\leq (1-\varepsilon_0)u^*(x),\quad \forall t\geq T_0,\ x\in\Omega.
\end{equation}

\end{lem}
\begin{proof}
Let $x_0\in\Omega^*$ and set 
$$\underline{v}(x_0):=\liminf_{t\to\infty}v(x_0,t;u_0,v_0) .$$
By Lemma \ref{lemma3} and the comparison principle for ODEs,  it holds that 
\begin{equation}
a(x_0)\geq \underline{v}(x_0)\geq a(x_0)-u^*(x_0)>0.
\end{equation}

Let $\mu>0$ so that 
$$ 
m_\mu:=\min\{a(x)-u^*(x), \ |x-x_0|\leq \mu\}>0.
$$
By \eqref{eq:03}, for every $\beta>1$, there is $t_\beta\gg 1$ so that 
$$ 
v(x,t;u_0,v_0)\geq \frac{1}{\beta}m_\mu, \quad \forall\ t\geq t_\beta,\quad |x-x_0|\leq \mu.
$$
So we can perturb $a(x)$ around $x_0$ and get a function $a_\beta(x)$ satisfying
$$ 
\begin{cases}
a(x)-v(x,t;u_0,v_0)\leq a_\beta(x),\quad \forall x\in\Omega, t\geq t_\beta;\cr
 a_{\beta}(x)\leq a(x),\quad \forall x\in\Omega;
 \cr
 a_{\beta}(x)=a(x),\quad |x-x_0|\geq \mu; \cr
a_\beta(x_0)<a(x_0).
\end{cases}
$$
Let $u^*_\beta$ denote the unique  positive solution of \eqref{eq:steady-state} with $a(x)$ being replaced by $a_\beta(x)$.  Hence, it holds that 
\begin{equation}\label{A:1}
u(x,t;u_0,v_0)\leq u_{\beta}(x,t;u(\cdot,t_\beta,u_0,v_0)),\quad \forall\ t\geq t_\beta,
\end{equation}
where $ u_{\beta}(x,t;u(\cdot,t_\beta,u_0,v_0))$ denotes the classical solution of \eqref{Main-eq2}, with $a(x)$ being replaced by $a_{\beta}$, satisfying $u_{\beta}(x,t_\beta;u(\cdot,t_\beta,u_0,v_0))=u(x,t_\beta,u_0,v_0)$.  Similarly, as in the proof of Lemma \ref{lemma3}, we have that 
$$ 
\liminf_{t\to\infty}v(x,t;u_0,v_0)\geq [a(x)-u^*_\beta(x)]_+, 
\quad \forall\ x\in\Omega.
$$
But since $u^*_\beta(x)$ is a sub-solution of \eqref{eq:steady-state}, by the 
comparison principle, it holds that 
\begin{equation}\label{A:2}
 u^*_\beta(x)<u^*(x),\quad \forall \ x\in\overline{\Omega}.
\end{equation}
Hence
$$ 
\underline{v}(x_0)\geq a(x_0)- u^*_\beta(x_0)>a(x_0)- u^*(x_0).
$$
Now, observe from \eqref{A:1} that 
$$ 
\limsup_{t\to\infty}\sup_{x\in\Omega}[u(x,t;u_0,v_0)-u^*_{\beta}(x)]\leq 0.
$$
This together with \eqref{A:2} yield the last assertion of the Lemma.
\end{proof}

\begin{tm}\label{T:1} Suppose that $u^*(x)>0$. Let $\varepsilon_0>0$ be given by Lemma {\rm\ref{Lem1}}.
It holds that 
\begin{equation}\label{d:1}
\liminf_{t\to\infty}v(x,t;u_0,v_0)\geq [a(x)-(1-\varepsilon)u^*(x)]_+, \quad \text{uniformly in }\ x\in\Omega
\end{equation}
for every $0\leq \varepsilon< \varepsilon_0$.
\end{tm}
\begin{proof}
Let $0\leq\varepsilon<\varepsilon_0$ be fixed. Note that it is enough to show that \eqref{d:1} holds on the set $ \Omega_{\varepsilon}:=\{x\in\Omega \ :\ a(x)\geq (1-\varepsilon)u^*(x)\}$. For, by Lemma \ref{Lem1}, we have that 
$$ 
v_{t}\geq v(a(x)-(1-\varepsilon_0)u^*-v),\quad \forall\ t\ge T_0, x\in\Omega.
$$
Hence, by the comparison principle for ODEs, we conclude that 

\begin{align}\label{d:2}
& v(x,t;u_0,v_0)\geq\cr &\frac{1}{e^{-(a(x)-(1-\varepsilon_0)u^*(x))(t-T_0)}\left[\frac{1}{v(x,T_0;u_0,v_0)}-\frac{1}{(a(x)-(1-\varepsilon_0)u^*(x))}\right]+\frac{1}{a(x)-(1-\varepsilon_0)u^*(x)}}\cr
\end{align}
for every $x\in\Omega_{\varepsilon}$ and $t\geq T_0$. Observe that 
$$ 
\sup_{x\in\Omega_{\varepsilon}}\left| \frac{1}{v(x,T_0;u_0,v_0)}-\frac{1}{(a(x)-(1-\varepsilon_0)u^*(x))}\right|<\infty
$$
and 
\begin{align*} 
\sup_{x\in\Omega_\varepsilon}e^{-(a(x)-(1-\varepsilon_0)u^*(x))(t-T_0)}\leq & \sup_{x\in \Omega_\varepsilon}e^{-(\varepsilon_0-\varepsilon)u^*(x)(t-T_0)}\cr
\leq & e^{-(\varepsilon_0-\varepsilon)u_{\inf}^*(t-T_0)}\to 0 \quad \text{as}\ t\to \infty.
\end{align*}
Hence, we conclude that the expression at the right hand side of the inequality \eqref{d:2} converges to $a(x)-(1-\varepsilon_0)u^*(x)$ uniformly on $\Omega_{\varepsilon}$, which combined with inequality \eqref{d:2} and the fact that $\varepsilon_0>\varepsilon$ lead to \eqref{d:1}.
\end{proof}

\begin{lem}\label{Lem2}
Let $\{u^*_{k}\}_{k\geq 0}$ be the sequence of Theorem {\rm\ref{Tm-appriori-result}}. Then 
 for every $k\ge 0$ such that $u_{k}^{*}(x)>0$ there is $\varepsilon_k>0$ such that 
$ \liminf_{t\to\infty}v(t,x;u_0,v_0)\geq [a(x)-(1-\varepsilon_k)u^*_k(x)]_{+}$ uniformly in  $x\in\Omega$.
\end{lem}
\begin{proof}
 Let $k\ge0$ such that $u^*_k(x)>0$.  If $k=0$, the result follows from Theorem \ref{T:1}. So, we may suppose that $k\geq 1$. Let 
$$ 
K=\max\{j: 0\leq j\leq k\,\, \textrm{such that the lemma holds} \}.
$$
We will show that $K=k$. Suppose not. Hence $0\leq K\leq k-1$. 
Since $k>K$, by Theorem \ref{Tm-appriori-result}(i) we note that $\min_{x\in\overline{\Omega}}u^*_{K}\ge \min_{x\in\overline{\Omega}}u^*_{k}>0$. Now, by induction hypothesis,  for every $0\leq\varepsilon<\varepsilon_{K}$, there is $T_{\varepsilon}\gg 1$ such that 
$$ 
u_t\leq d \Delta u +u(a-\left[a(x)-(1-\varepsilon)u_K^*(x)\right]_{+}-u),\quad \forall\ t\geq T_{\varepsilon}.
$$
Let $u^*_{K+1,\varepsilon}(x)$ denotes the unique non-negative attracting solution of  
\begin{equation}\label{d:3-varepsilon}
\begin{cases}
u_t=d \Delta u +u(a(x)-\left[a(x)-(1-\varepsilon)u^*_K(x)\right]_+-u),\quad & x\in\Omega,\ t>0,\cr
\frac{\partial u}{\partial n}=0,\quad & x\in\partial\Omega,\ t>0.
\end{cases}
\end{equation}
Note that $u^*_{K+1,\varepsilon}(x)$ is a non-negative steady state  solution of \eqref{d:3-varepsilon}. 
Since $u(x,t;u_0,v_0)$ is a sub-solution of \eqref{d:3-varepsilon} for $t\ge T_{\varepsilon}$, we then conclude that 
\begin{align}\label{d:24}
&\limsup_{t\to\infty}\sup_{x\in\Omega}[u(x,t;u_0,v_0)-u^*_{K+1,\varepsilon}(x)] \leq 0.
\end{align}
 It is clear that $u^*_{K+1,\varepsilon}(x)<u^*_{K+1}(x)$ for every $x\in\overline{\Omega}$. Hence, by \eqref{d:24},  there exist $\tilde{\varepsilon}_{K+1}>0$ and $T_{K+1}\gg 1$ such that 
\begin{equation}\label{d:7}
u(x,t;u_0,v_0)\leq (1-\tilde{\varepsilon}_{K+1})u^*_{K+1}(x),\quad \forall t\ge T_{K+1}, \ x\in\Omega.
\end{equation}
Observe that inequality \eqref{d:7} is equivalent to inequality \eqref{d:9}. Therefore, by the arguments of the proof of Theorem \ref{T:1}, we can show that there is $0<\varepsilon_{K+1}\ll 1$ such that
$$
\liminf_{t\to\infty}v(x,t;u_0,v_0)\geq [a(x)-(1-\varepsilon)u^*_{K+1}(x)]_+
$$
uniformly in $x\in\Omega$ for all $0\leq \varepsilon<\varepsilon_{K+1}$.  
Hence we must have that $K\geq K+1$, which is absurd. Therefore $K=k$.
\end{proof}

Now, we present the proof of Theorem \ref{Main-tm-1}.

\begin{proof}[Proof of Theorem {\rm\ref{Main-tm-1}}] We suppose  $a(x)$ satisfies hypothesis  \eqref{H2}.
We first note that when $u^*\equiv0$, then the result easily follows from Lemma \ref{lemma3}. So, we may suppose that $u^*>0$.  Hence by Theorem \ref{Tm-appriori-result} (ii), we have that $\|u^*_k\|_{\infty}\to 0$ as $k\to\infty$. Thus, by Lemma \ref{Lem2}, we conclude that
$\liminf_{t\to\infty}v(x,t;u_0,v_0)\geq a_+(x)$ for every $x\in\Omega$. On the other hand, since $u(x,t;u_0,v_0)\geq 0$ for every $x\in\Omega$ and $t\geq 0$, it follows from the comparison principle for ODEs that $\limsup_{t\to\infty}v(x,t;u_0,v_0)\leq a_+(x)$ for every $x\in\Omega$. 
Hence, $\lim_{t\to\infty}v(x,t;u_0,v_0)=a_+(x)$
uniformly in $x$ and accordingly, 
$\lim_{t\to\infty}\|u(\cdot,t;u_0,v_0)\|_{\infty}=0$ uniformly in $x$. 
\end{proof}

\section{Dynamics of solutions of \eqref{Main-eq1} when $a(x)$ is strictly positive}

This section is devoted to the study of dynamics of system \eqref{Main-eq1} when $a(x)$ is strictly positive, i.e. $a_{\min}>0$. Thanks to Theorem \ref{Tm-appriori-result} and Lemma \ref{Lem2}, we have the following 
{\it a priori} estimate on the solutions of \eqref{Main-eq1}:

\begin{tm}\label{Tm-02-0} Suppose that $a_{\min}>0$. Then for every non-negative
and not identically zero 
initial condition $u_0(x),v_0(x)\in C(\overline{\Omega})$ satisfying
\begin{equation}\label{eq:initial-cond2}
   \{x\in\overline{\Omega}\ :\ v_0(x)=0\}\subset \{x\in\overline{\Omega} \ : \ a(x)=a_{\min}\},
\end{equation}
we have that 
\begin{equation}\label{asymptotic-eq-2}
\limsup_{t\to\infty}\sup_{x\in\Omega}u(x,t;u_0,v_0)\leq a_{\min},
\end{equation}
and 
\begin{equation}\label{asymptotic-eq-02}
\liminf_{t\to\infty}v(x,t;u_0,v_0)\geq [a(x)- a_{\min}]_{+}\quad \textrm{uniformly in }\ x\in\Omega.
\end{equation}
\end{tm}

Note that (\ref{eq:initial-cond2}) automatically holds
if $v_0(x)>0$ for all $x\in\bar\Omega$. In 
particular, Theorem \ref{Tm-02-0} applies when
$u_0$ is non-negative and $v_0>0$ in $\bar\Omega$.

Next, we find a sufficient condition on $(u_0,v_0)$ to ensure that the equality holds for \eqref{asymptotic-eq-2}. We start with the following lemma.

\begin{lem}\label{lemma:monotonicity of Lyapunov function} Suppose that $0<\min\{u_{0\min},v_{0\min}\}$.
Let $(u,v)(x,t;u_0,v_0)$ be the classical solution of \eqref{Main-eq1} and define 

\begin{equation}\label{eq:Lyapunov-function}
\mathcal{M}(t):=\int_{\Omega}\ln(\frac{v}{u})dx.
\end{equation}
Then 
\begin{equation}\label{eq:monotonicity-lyapunov-function}
\frac{d}{dt}\mathcal{M}(t)=-d\||\nabla \ln u(\cdot,t;u_0,v_0)|\|_{L^2(\Omega)}^2.
\end{equation}
Hence,
\begin{equation}\label{eq:monotonicity-lyapunov-function-1}
\mathcal{M}(t)=\mathcal{M}(0)-d\int_0^t\||\nabla \ln u(\cdot,s;u_0,v_0)|\|_{L^2(\Omega)}^2 ds,\quad \forall t\ge 0.
\end{equation}
\end{lem}
\begin{proof}

Notice form \eqref{Main-eq1} that 
$$ 
\partial_t\ln(\frac{v}{u})=\frac{v_t}{v}-\frac{u_t}{u}=-d\frac{\Delta u}{u}
$$
Hence,  integrating with respect to the space variable yields
\begin{align*} 
\frac{d}{dt}\int_{\Omega}\ln(\frac{v}{u})dx=-d\int_{\Omega}\frac{\Delta u}{u}dx
&=d\int_{\Omega}<\nabla(\frac{1}{u}),\nabla u>dx 
=-d\int_{\Omega}\frac{|\nabla u|^2}{u^2}
dx.
\end{align*}
 Hence  the lemma holds.
\end{proof}

We introduce the following definition:
\begin{equation}\label{M-def}
\mathcal{M}_{\#}=\begin{cases}
\int_{\Omega}\ln(\frac{a(x)-a_{\min}}{a_{\min}})dx \quad  \text{if}\,\, \ln(\frac{a(x)-a_{\min}}{a_{\min}})dx\in L^1(\Omega;cr
-\infty & \text{otherwise}.
\end{cases}
\end{equation}

\begin{tm}\label{tm-2}
Suppose that  $0<\min\{u_{0\min},v_{0\min}\}$.  If 
\begin{equation}\label{initial-cond-1}
\mathcal{M}(0)=\int_{\Omega}\ln(\frac{v_0(x)}{u_0(x)})dx \leq \mathcal{M}_{\#},
\end{equation}
then
\begin{equation}\label{asymptotic-eq-3}
\limsup_{t\to\infty}\sup_{x\in\Omega}u(x,t;u_0,v_0)= a_{\min}.
\end{equation}
\end{tm}

\begin{proof} Suppose to the contrary that \eqref{asymptotic-eq-3} is false. Then by Theorem \ref{Tm-02-0}, there exist $\epsilon\in (0, a_{\min})$ and $T_\varepsilon\gg 1$ such that 
$$ 
u(x,t;u_0,v_0)<a_{\min}-\varepsilon\,\, \text{and} \,\, v(x,t;u_0,v_0)\geq a(x)-a_{\min}+\varepsilon\quad \forall\ x\in\Omega, \, t\ge T_{\varepsilon}.
$$
Hence
$$ 
\int_{\Omega}\ln(\frac{a(x)-a_{\min}+\varepsilon}{a_{\min}-\varepsilon})dx\leq \mathcal{M}(t),\quad \forall \ t\ge T_{\varepsilon},
$$
which combined with \eqref{eq:monotonicity-lyapunov-function-1} yield that,   for $t\geq T_{\varepsilon}$, 
\begin{equation}\label{d:12}
\int_{\Omega}\ln(\frac{a(x)-a_{\min}+\varepsilon}{a_{\min}-\varepsilon})dx\leq \mathcal{M}(t)=\mathcal{M}(0)-d\int_0^t\||\nabla \ln u(\cdot,s;u_0,v_0)|\|_{L^2(\Omega)}^2 ds.
\end{equation}
Observe that
$$ 
\frac{a(x)-a_{\min}+\varepsilon}{a_{\min}-\varepsilon}<\frac{a(x)-a_{\min}+\tilde{\varepsilon}}{a_{\min}-\tilde{\varepsilon}},\quad \forall x\in\Omega, \ \varepsilon<\tilde{\varepsilon}<a_{\min}.
$$
Hence, it follows from \eqref{d:12} that 
$\mathcal{M}_{\#}<\mathcal{M}(0),
$
contradicting \eqref{initial-cond-1}. Thus we must have that \eqref{asymptotic-eq-3} holds.
\end{proof}

\begin{rk}\label{rk2} We note the collection of functions $(U_c(x),V_c(x))=(c,a(x)-c)$ with $0< c< a_{\min}$,   forms a continuum of positive steady states of \eqref{Main-eq1}. In particular, for $(u_0(x),v_0(x))=(a_{\min},a(x)-a_{\min})$, we have that $\mathcal{M}(0)=\mathcal{M}_{\#}$ and $u(x, t;u_0,v_0)=a_{\min}$ for all $t\geq 0$ and $x\in\Omega$.
\end{rk}
 
 For $(x_i,y_i)\in\R^2$  we define the partial order 
 $$ 
(u_1,v_1)\precsim (u_2,v_2)\Leftrightarrow u_1\leq u_2\ 
\text{and}\ v_1\geq v_2.
 $$
 As a consequence of Remark \ref{rk2} and Theorem \ref{Tm-02-0}, we obtain the following result.
 
 \begin{tm}\label{tm-positive} Suppose that $(u_0,v_0)$ satisfy \eqref{eq:initial-cond2}. If in addition, $(u_0(x),v_0(x))$ satisfy 
 \begin{equation}\label{Initial-cond-1}
 (a_{\min},a(x)-a_{\min})\precsim (u_0(x),v_0(x)),\quad \forall\ x\in\Omega,
 \end{equation}
 then 
 \begin{equation}\label{asymp-lim-1}
 \lim_{t\to\infty}(u(x,t;u_0,v_0),v(x,t;u_0,v_0))=(a_{\min},a(x)-a_{\min}),\quad \forall x\in\Omega.
 \end{equation}
 \end{tm} 
 
 \begin{proof} By the comparison principle for competitive systems, it holds that 
\begin{equation*}
(a_{\min},a(x)-a_{\min})\precsim (u(x,t;u_0,v_0),v(x,t;u_0,v_0)),\quad \forall\ t\geq 0, x\in\Omega.
\end{equation*}
Thus, the result follows from Theorem \ref{Tm-02-0}.
 \end{proof}

We conclude with some comments on Theorems \ref{tm-2}
and \ref{tm-positive}. 
We first note that the hypotheses of Theorem \ref{tm-positive} implies that $\{x\in\overline{\Omega} \ :\ v_0(x)=0\}=\{x\in\overline{\Omega}\ :\ a(x)=a_{\min}\}$. In Theorem \ref{tm-2}, the hypothesis that $0<\min\{u_{0\min},v_{0\min}\}$ was essential to justify that $\mathcal{M}(t)$ is well defined by noticing that $0<\min\{u_{\min}(t),v_{\min}(t)\}$ for every $t\ge0$. Note that the right hand side of equation \eqref{eq:monotonicity-lyapunov-function-1} clearly suggests that $\mathcal{M}(t)$ might be well defined under a more general weaker assumption.   Now, if $\mathcal{M}(0)$ is well defined, then \eqref{Initial-cond-1} implies that $\mathcal{M}(0)\leq \mathcal{M}_{\#}$. Thus, we clearly see that these two theorems only complement each other and one does not imply the other. Moreover, 
they both provide sufficient conditions for the first species $u(x,t)$ to reach its possible maximum state at infinity when \eqref{eq:initial-cond2} holds.

\end{document}